\documentclass{amsart}  
\usepackage{graphicx}
\usepackage[dvipsnames]{xcolor}

\usepackage{amssymb,latexsym,amsfonts,amsmath}
\usepackage{graphicx}
\include{diagrams}
\usepackage{eso-pic}
\usepackage{color}
\usepackage{type1cm}

\usepackage{amssymb,latexsym,amsfonts,amsmath}
\usepackage{cite}
\usepackage{textcomp}
\usepackage{algorithmic, algorithm}
\usepackage[utf8]{inputenc}
\usepackage{enumitem}
\usepackage{hyperref}

\usepackage{mathdots}
\allowdisplaybreaks
\newcounter{MYtempeqncnt}

\newcommand{\norm}[1]{\lVert#1\rVert}

\newtheorem{theorem}{Theorem}[section]

\newtheorem{lemma}[theorem]{Lemma}

\theoremstyle{definition}

\theoremstyle{remark}
\newtheorem{remark}[theorem]{Remark}
\numberwithin{equation}{section}

\renewcommand{\Re}{{\mathbb{R}}}
\newcommand{\R}{{\mathbb{R}}}

\newcommand{\N}{{\mathbb{N}}}

\begin{document}


\title{Dirty derivatives for output feedback stabilization}

\author[Matteo~Marchi]{Matteo Marchi}
\address{Department of Electrical and Computer Engineering\\ 
University of California at Los Angeles\\
Los Angeles, CA 90095-1594, USA}
\email{matmarchi@ucla.edu}

\author[Lucas~Fraile]{Lucas Fraile}
\address{Department of Electrical and Computer Engineering\\ 
University of California at Los Angeles\\
Los Angeles, CA 90095-1594, USA}
\email{lfrailev@ucla.edu}

\author[Paulo~Tabuada]{Paulo Tabuada}
\address{Department of Electrical and Computer Engineering\\ 
University of California at Los Angeles\\
Los Angeles, CA 90095-1594, USA}
\urladdr{http://www.ee.ucla.edu/$\sim$tabuada}
\email{tabuada@ee.ucla.edu}

\thanks{The first two authors contributed equally to this paper.}%
\thanks{This work was supported in part by the CONIX Research Center, one of six centers in JUMP, a Semiconductor Research Corporation (SRC) program sponsored by DARPA.}

\begin{abstract}
Dirty derivatives are routinely used in industrial settings, particularly in the implementation of the derivative term in PID control, and are especially appealing due to their noise-attenuation and model-free characteristics. In this paper, we provide a Lyapunov-based proof for the stability of linear time-invariant control systems in controller canonical form when utilizing dirty derivatives in place of observers for the purpose of output feedback. This is, to the best of the authors' knowledge, the first time that stability proofs are provided for the use of dirty derivatives in lieu of derivatives of different orders. In the spirit of adaptive control, we also show how dirty derivatives can be used for output feedback control when the control gain is unknown. 
\end{abstract}

\maketitle

\section{Introduction}
\label{sec:Introduction}

In the field of control, the need for the computation or estimation of time derivatives of a signal is a common occurrence in both theoretical and practical works, especially in the design and implementation of output-feedback controllers. An example being the ubiquitous PID controller, that explicitly includes a (sometimes omitted) term proportional to the derivative of the error signal.

While the task of differentiating a signal is conceptually simple, naive approaches are often inadequate for purposes of control. Problems often encountered are the amplification of noise due to the ill-posedness of the task \cite{engl1996regularization,lu2013regularization,KhalilNoiseHighGain,chartrand2011numerical}, and non-causality of an ideal differentiating filter \cite{oppenheim1997signals}. The approach commonly adopted to address both these issues involves calculating the derivative $\dot y$ of a signal $y$ in an approximate way by computing a low-pass filtered version of $\dot y$. In the context of PID control, this is presented as the addition of a pole to the transfer function of the derivative operator \cite{Astrom, AstromBook,PIDanalysis}. This notion of approximate derivative of $y$ is what we refer to as ``dirty derivative'' of $y$.

Despite the pervading use of these dirty derivatives in close-loop control \cite{Nunes,Astrom, AstromBook,PIDanalysis}, theoretical foundations have only been recently established in Loria's seminal work~\cite{Loria}, where a controller is designed by employing dirty derivatives as approximations for unavailable state measurements, dispensing with the need of designing an observer. With this controller, Loria is able to demonstrate how to asymptotically stabilize Euler-Lagrange plants by output feedback, solving a $25$-year-old open problem in control theory. Since then, more authors have resorted to dirty derivatives within output feedback~\cite{Sira,Miranda,Furtat,Jaafar}. Note that for systems of relative degree higher than 2, the output feedback controllers proposed in Loria's work \cite{Loria} require a combination of dirty derivatives and an observer, see Remark 5 in \cite{Loria}.

Dirty derivatives offer several advantages over other techniques to estimate derivatives such as algebraic  de-noising and derivative estimation approaches \cite{AlgebraicDDs} and  low-power peaking-free  high-gain  observers \cite{PFHighGain}. The first advantage is the simplicity of their design: only one parameter needs to be designed and no knowledge about the plant is required. The second advantage is robustness with respect to measurement noise. This is illustrated in Section \ref{sec:Examples} where we show, through numerical simulations, that dirty derivatives outperform the algebraic de-noising and derivative estimation approach and the low-power peaking-free high-gain observers. Motivated by the extensive practical applications, lack of theoretical guarantees, and the previously described advantages with respect to alternative techniques, we provide sufficient conditions for asymptotic stabilization of linear systems in controller form when dirty derivatives are used to approximate derivatives of arbitrary order. We also address the important case where the control gain is unknown.

\section{Preliminaries}
\label{sec:Preliminaries}
Let $x_1: \Re_{\geq 0} \to \Re$ be a smooth function, we call $\widehat x_2: \Re_{\geq 0} \to \Re$ a first-order dirty derivative of $x_1$ if the Laplace transforms of $\widehat x_2$ and $x_1$, respectively $\widehat X_2$ and $X_1$, satisfy the relationship:
\begin{equation}
  \widehat X_2(s) = \frac{\sigma s}{s + \sigma} X_1(s),  
\end{equation}
for some positive $\sigma \in \Re^+$, under the assumption of zero initial conditions. Identifying $sX_1$ as the Laplace transform of $\dot{x}_1$, it is straightforward to see that dirty derivatives can be interpreted as the output of a low-pass filter with input $\dot{x}_1(t)$ and a pole set at $\sigma$. This portrays one of the two properties that make dirty derivatives appealing; dirty derivatives provide a filtered version of the derivative of a signal, providing robustness against measurement noise. The other main property that makes dirty derivatives particularly interesting is that they can be computed through the state space representation:
\begin{equation}
\label{eq:dirtyderivative1}
\begin{aligned}
\dot{q}_1 = -\sigma (q_1 + \sigma x_1), \quad 
\widehat x_2 = q_1 + \sigma x_1.
\end{aligned}
\end{equation}
This implies we can obtain $\widehat x_2$ solely based on measurements of $x_1$, providing a low-gain approach for approximating derivatives. For notational ease, in the rest of the paper we resort to a slight abuse of notation and define $\widehat x_1$ to be $x_1$. 
We can now introduce higher order dirty derivatives, which we recursively define as:
\begin{equation}
\label{eq:dirtyderivative2}
\begin{aligned}
\dot{q}_i = -\sigma (q_i + \sigma \widehat x_i), \quad \widehat x_{i+1} = q_i + \sigma \widehat x_i.
\end{aligned}
\end{equation}
Note that under this representation, the dirty derivative's own derivatives become:
\begin{equation}
\begin{aligned}
\dot{\widehat x}_{i+1} &= \dot{q}_i+ \sigma \dot{ \widehat x}_i 
= -\sigma (q_i + \sigma \widehat x_i) + \sigma \dot{ \widehat x}_i = -\sigma (\widehat x_{i+1}-\dot{ \widehat x}_i). 
\end{aligned}\notag
\end{equation}
The following analysis is based on the latter expression, yet it is important to keep in mind that dirty derivatives are implemented using equations \eqref{eq:dirtyderivative1}.
\section{Problem Statement}
\label{Sec:Models}
We consider a single-input single-output linear time-invariant system described by:
\begin{equation}
\label{eq:system1}
\begin{aligned}
\dot{x}=Ax + Bu, \quad 
y={}Cx, 
\end{aligned}
\end{equation}
where $A \in \R^{n \times n}$, $B \in \R^{n \times 1}$, and $C \in \R^{1 \times n}$ are known matrices and by $x\in \R^n$, $u\in \R$, $y\in \R$ we denote the state, input and output, respectively. Furthermore, we assume the matrices $A$ and $B$ to be in controller normal form:
\begin{equation}
\begin{aligned}
 A &= \begin{bmatrix}
    0_{(n-1)\times 1} & I_{n-1}\\
    \times & \times_{1\times (n-1)}
    \end{bmatrix},\quad B = \begin{bmatrix}
    0_{(n-1)\times 1}\\
    \times
    \end{bmatrix},
\end{aligned}
\end{equation}
where we denote the last row of $A$ by $A_n$ and the last element of $B$ by $B_n\ne 0$. Finally, we assume $C$ to be given by $C = \begin{bmatrix}1 &0_{(n-1)\times 1}
    \end{bmatrix}.$
    
Noting that these assumptions imply the pair (A,B) is controllable, we conclude that there exists a stabilizing controller $u(x) = Kx$ and a symmetric positive definite matrix $P$ such that:
\begin{equation}
\label{eq:riccati}
(A+BK)^TP+P(A+BK)=-Q,
\end{equation}
is satisfied for some symmetric positive definite matrix $Q$. Given that we do not have access to measurements of the state $x$, we adopt the controller:
\begin{equation}
\label{eq:ddcontroller}
u=K\hat{x},
\end{equation}
instead, where $\hat{x}=(\hat{x}_1,\hdots,\hat{x}_n)$ and $\hat{x}_i$ is the dirty derivative approximation of $x_i=\frac{d^{i-1}}{dt^{i-1}}y$.

Our goal is to provide theoretical guarantees under which global asymptotic stability is preserved when using approximates provided by dirty derivatives in place of the state. It should be noted that as dirty derivative don't necessarily converge to the actual derivatives, the well known separation principle cannot be used here.

\section{General Output Feedback}
\label{sec:GeneralOutputFeedback}
The dynamical system \eqref{eq:system1} in closed-loop with the controller \eqref{eq:ddcontroller} and the dirty derivative approximations given by \eqref{eq:dirtyderivative2} results in the system:
\begin{equation}\label{MainSystem}
    \begin{cases}
        \dot x &= Ax+BK\begin{bmatrix}\widehat x_1 & \widehat x_2 & \dots & \widehat x_n\end{bmatrix}^T\\
        \dot{\widehat{x}}_1 &= -\sigma(\widehat x_1-x_1)\\
        \dot{\widehat{x}}_2 &= -\sigma(\widehat x_2-\dot{\widehat x}_1)\\
        &\vdots\\
        \dot{\widehat{x}}_n &= -\sigma(\widehat x_n-\dot{\widehat x}_{n-1}),
    \end{cases}
\end{equation}
where $n\in\N$ and $\sigma\in\R_{>0}$ is a design parameter. \\

\begin{remark}
One can choose not to filter measurements before computing the dirty derivative approximations by defining $\widehat{x}_1 = y$ in system \eqref{MainSystem}. If one does so, the theorem below follows through with only minor modifications to the proof. Filtering measurements before computing the dirty derivatives is a design choice left to the user where the need to counterbalance noise attenuation versus faster response time in the approximations comes into play. 
\end{remark}
With this system at hand we can now introduce our main result, whose proof can be be found in the appendix.
\begin{theorem}\label{theo:main}
    There always exists $\underline\sigma\in\R_{>0}$ such that for any $\sigma\in(\underline\sigma, \infty)$ the linear time invariant dynamical system~\eqref{MainSystem} is asymptotically stable.
\end{theorem}
\section{An Adaptive Control Extension}
\label{sec:AdaptiveFeedback}
An extension of particular interest to the authors, given their recent work in \cite{fraile2021datadriven}, are feedback linearized systems described by the dynamics:
\begin{equation}
\label{System2}
\begin{aligned}
\dot{x}{}={}Ax + \beta Bu, \quad y{}={}Cx, 
\end{aligned}
\end{equation}
where \begin{equation}
\begin{aligned}
    A =& \begin{bmatrix}
    0_{(n-2)\times 1} & I_{n-2}\\
    \times & \times_{1\times (n-2)}
    \end{bmatrix},\quad B = \begin{bmatrix}
    0_{(n-2)\times 1}\\
    1
    \end{bmatrix},\\
    C=&\begin{bmatrix}
    1 & 0_{(n-2)\times 1}
    \end{bmatrix},
    \quad x=\begin{bmatrix}x_1& x_2&\dots& x_{n-1}\end{bmatrix},
\end{aligned}\notag
\end{equation} and $\beta$ is an unknown nonzero constant of known sign. 

In order to stabilize this system we propose the use of a dynamic controller of the form:
\begin{equation}
\label{eq:dynamic_controller}
    \dot u = -\gamma\left(\dot x_{n-1}-K\begin{bmatrix}x_1& x_2&\dots& x_{n-1}\end{bmatrix}^T\right),
\end{equation}
where $\gamma \in \R$ satisfies $\mathrm{sign}(\gamma)=\mathrm{sign}(\beta)$ and $K \in \R^{n-1}$ is such that the equality:
\begin{equation}
(A+BK)^TP+P(A+BK)=-Q,\notag
\end{equation}
is satisfied for some symmetric positive definite matrices $P$ and $Q$, existence of which is guaranteed due to $(A,B)$ being a controllable pair. Without loss of generality we assume $\beta>0$ and thus $\gamma \in \R_{>0}.$ The motivation for the use of this controller is that it asymptotically enforces the equality $\dot x_{n-1} = Kx$ without the need of explicitly estimating $\beta$.

As before, due to only having access to measurements of $x_1$, we replace all derivatives of $x_1$ in \eqref{eq:dynamic_controller} with approximations provided through dirty derivatives. Defining $x_n$ as $\dot{x}_{n-1}$ results in the following close-loop dynamical system:
\begin{equation}\label{MainSystem2}
    \begin{cases}
        \dot x_1 &= x_2\\
        &\vdots\\
        \dot x_{n-1} &= x_{n}\\
        \dot x_{n} &=A_n \dot x_{1:n-1} -\beta\gamma\left(\widehat x_{n} - K\begin{bmatrix}\widehat x_1& \dots &\widehat x_{n-1}\end{bmatrix}^T\right)\\
        \dot{\widehat{x}}_1 &= -\sigma(\widehat x_1-x_1)\\
        \dot{\widehat{x}}_2 &= -\sigma(\widehat x_2-\dot{\widehat x}_1)\\
        &\vdots\\
        \dot{\widehat{x}}_{n} &= -\sigma(\widehat x_{n}-\dot{\widehat x}_{n-1}),
    \end{cases}
\end{equation}
where $A_n$ denotes the last row of $A$ and $x_{i:j}$ denotes the vector containing states $x_i$ to $x_j$, i.e., $x_{i:j}=\begin{bmatrix}x_i&x_{i+1}&\dots&x_{j}\end{bmatrix}^T$. We are now ready to introduce our second result, whose proof can be found in the appendix.
\begin{theorem}\label{theo:main_ext}
    There always exists $\underline\gamma\in\R_{>0}$, such that for any $\gamma\in(\underline\gamma, \infty)$ there exists $\underline\sigma\in\R_{>0}$ such that for any $\sigma\in(\underline\sigma, \infty)$ the LTI dynamical system~\eqref{MainSystem2} is asymptotically stable.
\end{theorem}

\section{Simulation Examples}
\label{sec:Examples}
In this section we provide two simulation examples to portray the potential benefits of using dirty derivative approximations instead of state estimates. The first example compares the dirty derivatives approximations with estimates provided by two state-of-the-art methods, a low-power peaking-free high-gain observer \cite{PFHighGain} and an algebraic de-noising and derivative estimation approach \cite{AlgebraicDDs}. The second example compares the closed-loop performance of a linear controller under measurement noise when dirty derivative approximations and state estimates are used in place of state measurements.

\subsection{Dirty derivatives as an observer}
We begin this subsection by reminding the reader that dirty derivatives provide approximations, not estimates, of the derivative of a signal. This is worth noting as one would expect the performance of any observer to outpace a dirty derivative based approximation. While this is true in the absence of noise, the low-pass filtering effect, low gain and estimation through integration aspects of dirty derivatives provide such robustness against noise that we found dirty derivatives able to compete with, and often outperform, state-of-the-art observers.

For the examples that follow we consider the recently developed low-power peaking-free high-gain observer \cite{PFHighGain}, which has been shown to outperform traditional low-power high-gain observers, and the algebraic derivative estimation method developed in \cite{AlgebraicDDs}. We reproduce the example presented in \cite{PFHighGain}, including observer gains and system's initial conditions. The system under consideration is given by:
\begin{equation}
 \begin{cases}
 \dot{x}_1 &= x_2 \\
\dot{x}_2 &= x_3 \\
\dot{x}_3 &= x_4 \\
\dot{x}_4 &= x_5 \\
\dot{x}_5 &= 0.2(x_1^2-1)-x_2-x_3-4x_4-x_5, 
 \end{cases}   \notag
\end{equation}
and the low-power peaking-free high-gain observer presented in equations (11a) in \cite{PFHighGain} has full knowledge of the dynamics. We obtain our dirty derivative approximations using equations \eqref{eq:dirtyderivative2} and algebraic derivative estimates using equations $(69)-(72)$ in \cite{AlgebraicDDs}. We chose the dirty derivative parameter $\sigma$ to be $5$, and the parameter $a$ of the method in \cite{AlgebraicDDs} to also be $5$.

While in the presence of noise dirty derivatives outperform both other methods in most of the estimates, we present only the first two derivatives to simplify the exposition. We believe it is important to emphasize to the reader the relatively simple implementation of dirty-derivatives when compared with the other two methods presented. The low-power peaking-free high-gain observer requires full knowledge of the system and a complicated process in order to select the gains and saturation limits, and the algebraic approach requires the computation of matrices $\bar{M}^{-1}$ and $C$, see equations $(72)$, $(70)$ and $(34)$ in \cite{AlgebraicDDs}.

Figure \ref{fig:SignalVsNoise} portrays the original signal $x_1$ overlaid onto the measured noisy signal obtained by adding measurement noise to $x_1$. We generate the measurement noise by filtering zero-mean Gaussian noise with a variance of $0.002$ and sampling time of $10^{-6}$ with a low-pass filter (with band $]0,200] Hz$). Figures \ref{fig:SignalVsEstimates} and \ref{fig:AbsoluteError} respectively compare the estimates and absolute estimation errors for the first and second derivatives of $x(t)$ as estimated under measurement noise by all three methods.

\begin{figure}[!h]
\centering
\includegraphics[width=0.9\textwidth]{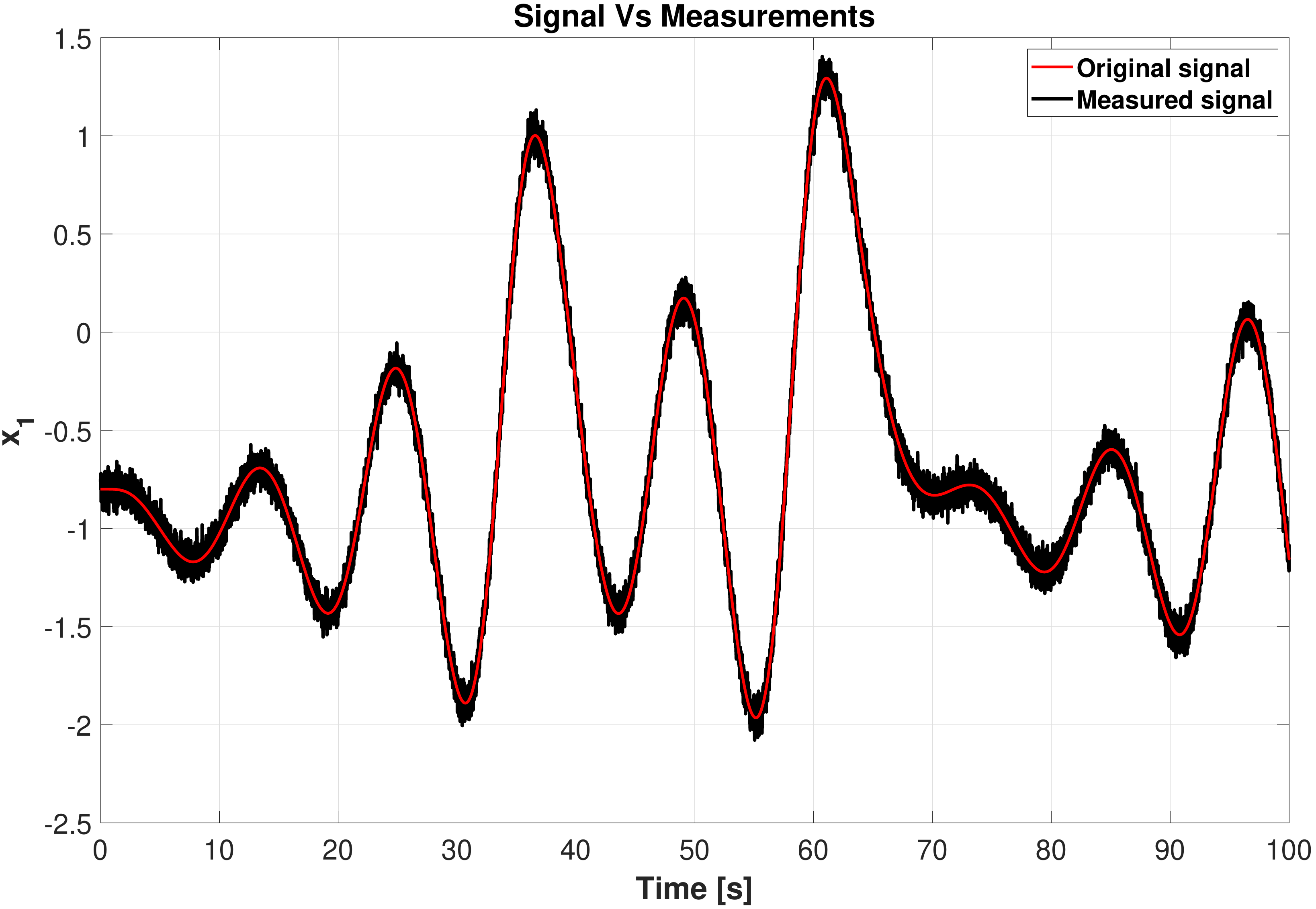}
\caption{Original signal overlaid onto the noisy signal created by adding measurement noise generated as filtered Gaussian noise in the band 0-200 Hz.}
\label{fig:SignalVsNoise}
\end{figure}

\begin{figure}[!h]
\centering
\includegraphics[width=0.9\textwidth]{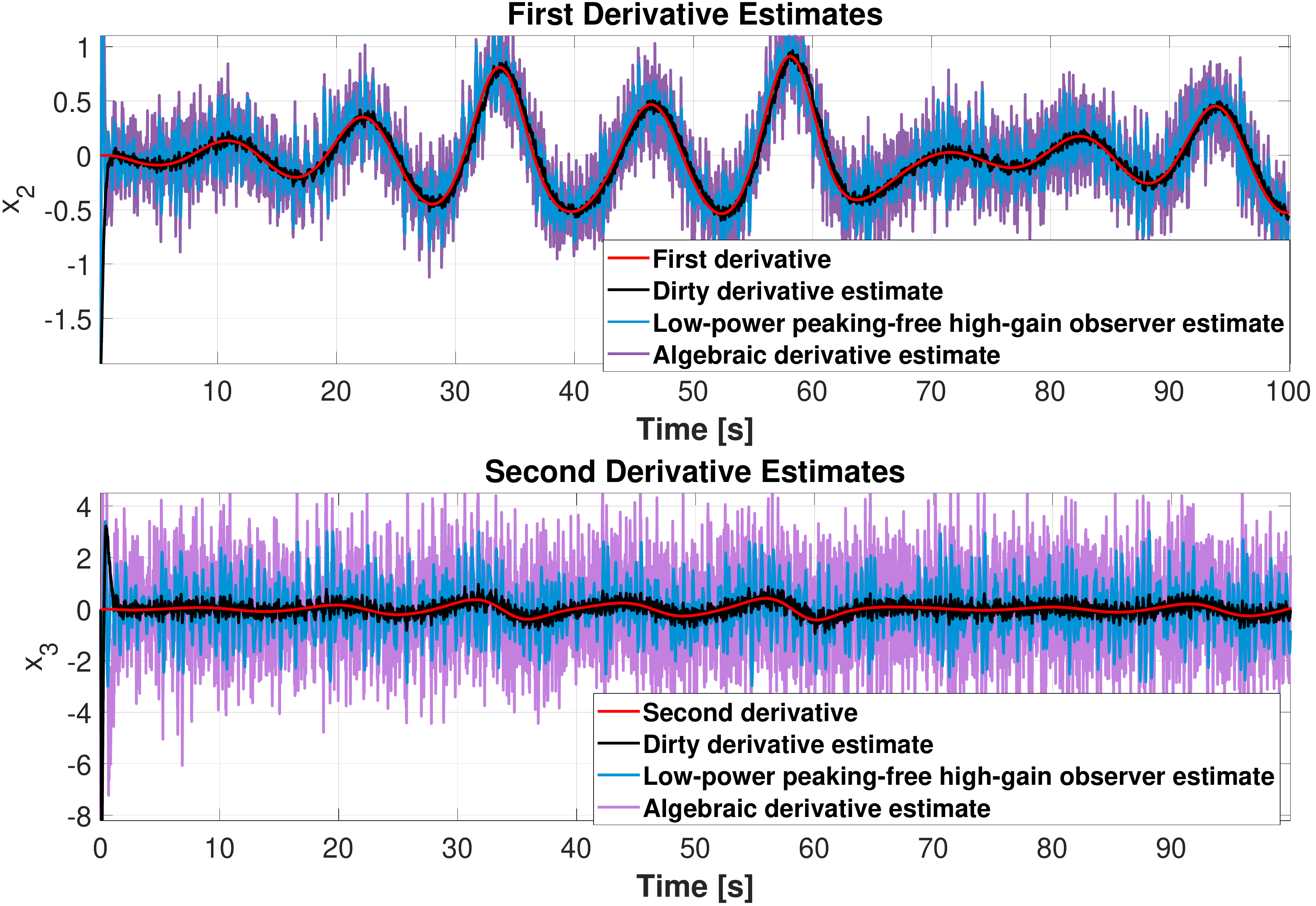}
\caption{Comparison between estimates for the first and second derivatives of the signal in Figure  \ref{fig:SignalVsNoise} provided by the dirty derivative method, the peaking-free low-power high-gain observer and the algebraic derivative method.}
\label{fig:SignalVsEstimates}
\end{figure}

\begin{figure}[!h]
\centering
\includegraphics[width=0.9\textwidth]{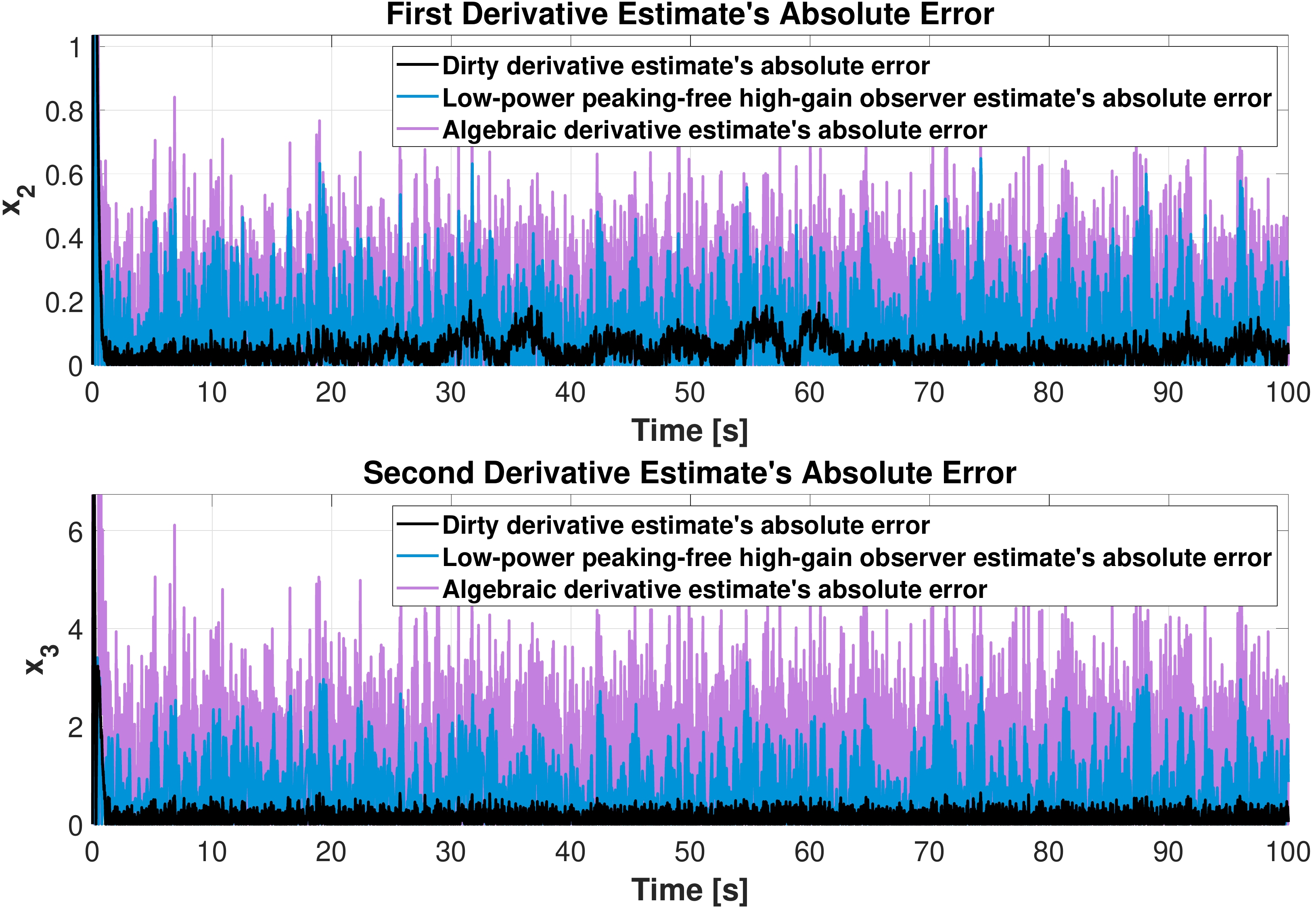}
\caption{Comparison between absolute estimation error for the first and second derivatives of the signal in Figure \ref{fig:SignalVsNoise} provided by the dirty derivative method, the peaking-free low-power high-gain observer and the algebraic derivative method.}
\label{fig:AbsoluteError}
\end{figure}

These figures underscore the potential robustness against measurement noise that dirty derivatives can offer. The next subsection emphasizes this point in a closed-loop example.

\subsection{Dirty-derivatives in closed-loop}
In this subsection we provide simulation examples for the results presented in Section \ref{sec:GeneralOutputFeedback}. Starting with the former, we consider system \eqref{eq:system1} with:
\begin{align} A_n = \begin{bmatrix} -1 & 3 & 3 & 0 & 1\end{bmatrix}, \ B_n = -4,\ x_0 = \begin{bmatrix} -1 & 5 & -1 & -3 & 3\end{bmatrix} \notag.
\end{align}
A large number of simulations were performed where each of the entries of the above entities were chosen randomly as integers in the interval $[-5,5]$ with uniform probability. Consistent performance was evidenced through these simulations and the above set of values was chosen as a good representation of the observed behavior.

We design the linear controller $u = Kx$ where $K$ is obtained by solving the linear quadratic regulator problem with costs $Q = C^TC$ and $R = 0.001$, and measurements satisfy the equation: $y = Cx+v,$ where $v$ is generated as in the previous section with a variance of $0.01$. We then implement $u=K\widehat{x}$ where $\widehat{x}$ is computed through: dirty derivatives as in \eqref{eq:dirtyderivative2} with $\sigma=75$; algebraic derivative estimation as in \cite{AlgebraicDDs} with $a=75$; and a high-gain observer as in section 2.2 in \cite{Khalil2002} where the poles of $H_0$ are set at $\{-1, -2, -3, -4, -5\}$ and $\varepsilon = \frac{1}{75}$. 

Figures \ref{SignalVsEstimatesCloseLoop} and \ref{AbsoluteErrorCloseLoop} respectively compare the trajectories and absolute trajectory tracking errors for these three close-loop systems. Note that all three methods achieve similar steady state performance despite the varying levels of complexity in their implementations. Regarding the transient, we observe both the dirty-derivative and the algebraic estimation methods to be less affected by peaking than the high-gain observer. 

\begin{figure}[!h]
\centering
\includegraphics[width=0.9\textwidth]{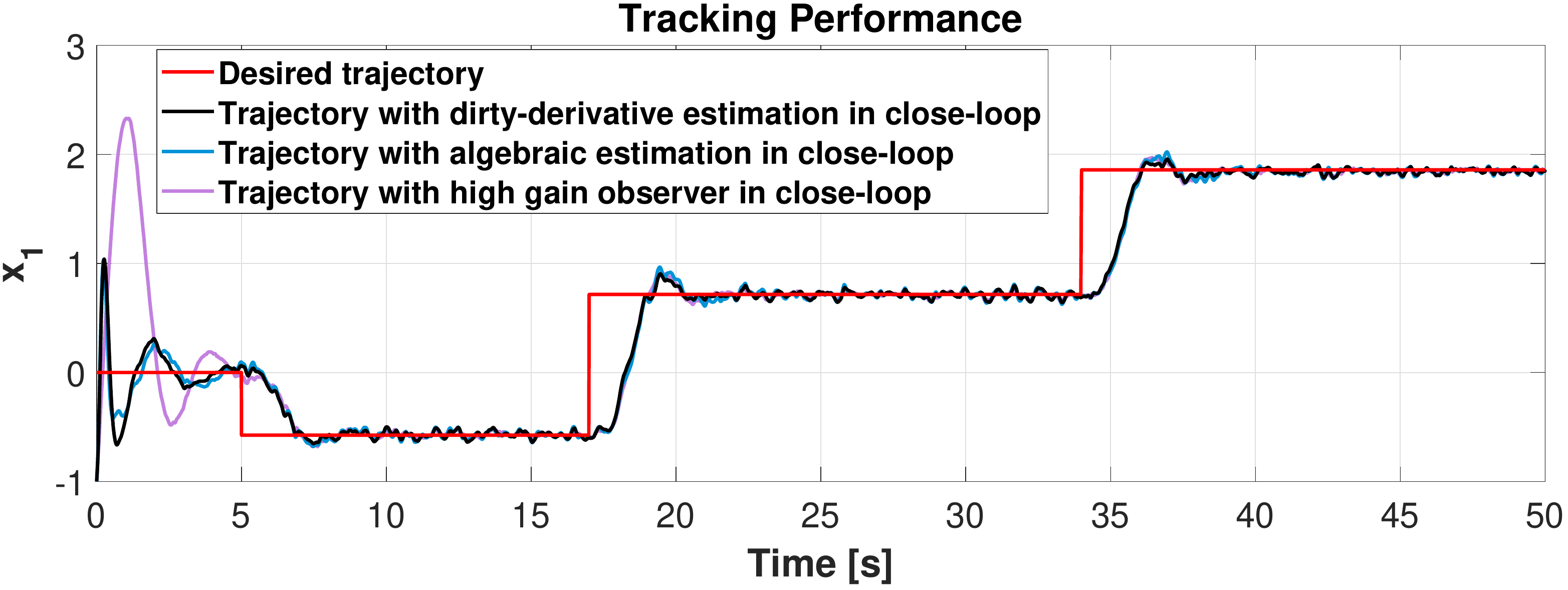}
\caption{Comparison between the desired trajectory and the achieved trajectories by the controller $u=K\widehat{x}$ when estimates are provided by the dirty derivative method, a high-gain observer and the algebraic derivative method.}
\label{SignalVsEstimatesCloseLoop}
\end{figure}

\begin{figure}[!h]
\centering
\includegraphics[width=0.9\textwidth]{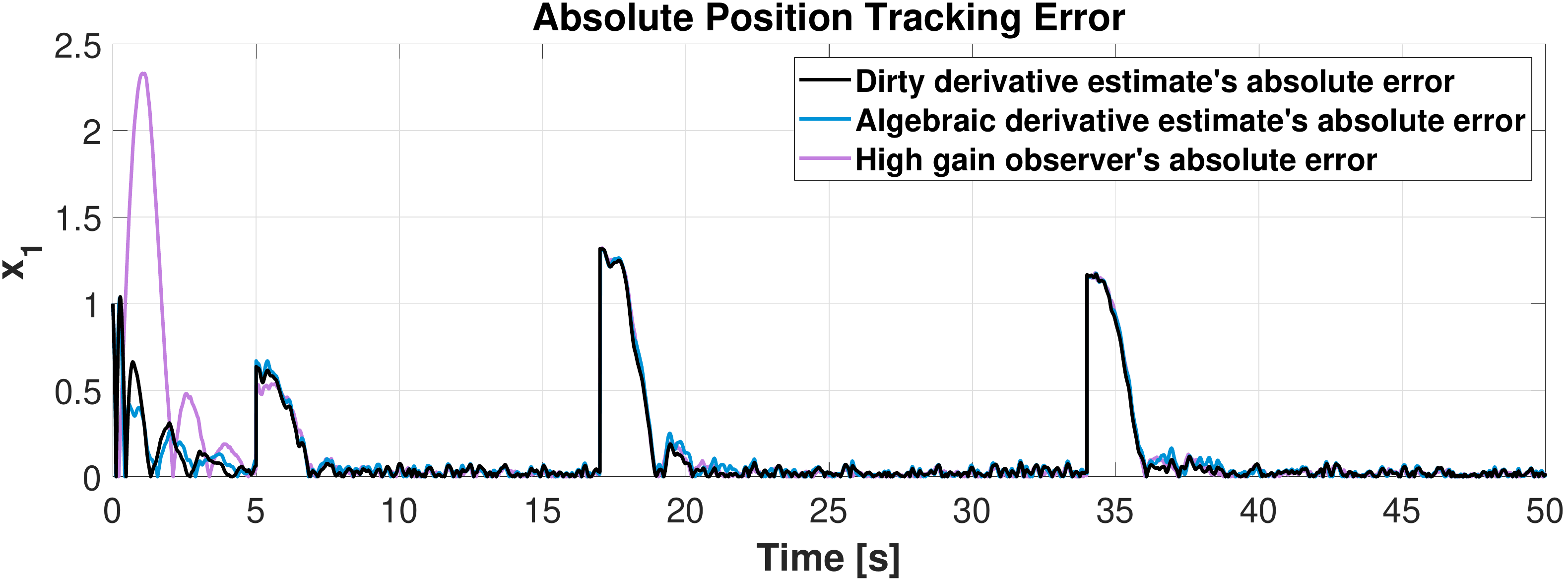}
\caption{Comparison of absolute tracking error incurred by the controller $u=K\widehat{x}$ when estimates are provided by the dirty derivative method, a high-gain observer and the algebraic derivative method.}
\label{AbsoluteErrorCloseLoop}
\end{figure}

\section{Conclusion}
In this paper we have provided two Lyapunov based proofs enabling the use of dirty derivatives in place of observers when performing output feedback for a large class of linear time invariant systems, i.e., arbitrary order systems in canonical control form. Being independent of the plant's model or inputs, having a single tuning parameter, low gains and noise attenuation characteristics, positions the dirty derivatives as an easy-to-use, straight forward estimation method. Future research involves their extension to general observable and controllable systems, along with investigating the relationships between dirty derivatives and the algebraic estimation method presented in \cite{AlgebraicDDs}.

\section{Appendix}
\subsection{Proof of Theorem \eqref{theo:main}}
\begin{theorem}
    There always exists $\underline\sigma\in\R_{>0}$ such that for any $\sigma\in(\underline\sigma, \infty)$ the linear time invariant dynamical system~\eqref{MainSystem} is asymptotically stable.
\end{theorem}
\begin{proof}
    We prove the above claim by a Lyapunov argument. First we introduce a set of $\frac{n(n+1)}{2}$ error variables that we index as:
    \begin{equation}
    \label{ErrorDef}
        \begin{aligned}
        \forall i\in\{1,2,\dots,n\},\forall j&\in\{0,1,\dots,n-i\}:\\
        &e_{i,j} = \widehat x_i^{(j)} - x_{i+j},      
    \end{aligned}
    \end{equation}
    where the superscript in parenthesis denotes the $j$-th derivative of a variable with respect to time. Note that the maximum value of $j$ is $n-i$, so that $i+j\leq n$ and the error term $e_{i,j}$ is always well-defined.
    
    To ease the exposition of our method, the newly defined error variables are organized in a matrix as follows, with the $i,j$-th entry denoting the error variable $e_{i,j-1}$:
    \begin{equation}\label{eq:errs}
        \begin{aligned}
        E &= \begin{bmatrix}
            e_{1,0} & e_{1,1} & \dots & e_{1,n-1}\\
            e_{2,0} & e_{2,1} & \dots &  0\\
            \vdots & \vdots & \iddots & \vdots\\
            e_{n,0} & 0 & \dots & 0
        \end{bmatrix}\\
        &= \begin{bmatrix}
            \widehat x_1-x_1 & \dot{\widehat x}_1-x_2 & \dots & \widehat x_1^{(n-2)}-x_n\\
            \widehat x_2-x_2 & \dot{\widehat x}_2-x_3 & \dots & 0\\
            \vdots & \vdots & \iddots & \vdots\\
            \widehat x_n-x_n & 0 & \dots & 0
        \end{bmatrix}.
        \end{aligned}
    \end{equation}
    
    Given \eqref{ErrorDef}, the time derivatives of the error variables are given by:
    \begin{equation}\label{eq:err_dot}
        \begin{aligned}
            \dot e_{i,j} = \dot{\widehat x}_i^{(j)}-\dot x_{i+j} &= -\sigma\left(\widehat x_i-\dot{\widehat x}_{i-1}\right)^{(j)}-\dot x_{i+j}\\
            &= -\sigma(\widehat x_i^{(j)}-\widehat x_{i-1}^{(j+1)})-\dot x_{i+j},
        \end{aligned}
    \end{equation}
    where the second equality is obtained by substituting $\dot{\widehat x}_i$ with the corresponding right-hand side in the equations of the dynamical system~\eqref{MainSystem}. Furthermore, note that we have defined $\dot{\widehat x}_0 = x_1$ for conciseness. We now introduce a set of properties that these error derivatives possess that are instrumental in the rest of the proof:
    \begin{enumerate}
        \item If $i=1$:
        \begin{equation}\notag
            \begin{aligned}
                \widehat x_i^{(j)}-\widehat x_{i-1}^{(j+1)} &= \widehat x_1^{(j)}-\widehat x_{0}^{(j+1)} = \widehat x_1^{(j)}-x_{1}^{(j)} \\
                &= \widehat x_1^{(j)}-x_{1+j} = e_{1,j}.
            \end{aligned}
        \end{equation}
        \item If $i>1$:\\
        We notice from~\eqref{ErrorDef} that $e_{i-1,j+1} = \widehat x_{i-1}^{(j+1)}-x_{i+j}$. Then:
        \begin{equation}
            \begin{aligned}
            \widehat x_i^{(j)}-\widehat x_{i-1}^{(j+1)} &= \widehat x_i^{(j)}-e_{i-1,j+1}-x_{i+j} \\ &= e_{i,j}-e_{i-1,j+1}.
            \end{aligned}\notag
        \end{equation}

        \item If $i+j<n$:\quad $$\dot x_{i+j} = x_{i+j+1}.$$
        \item If $i+j=n$:
        \begin{equation}
            \begin{aligned}
                \dot x_{i+j} & =A_n x +B_n K\begin{bmatrix}\widehat x_1&\widehat x_2&\dots&\widehat x_n\end{bmatrix}^T\\
                &=A_nx \\
                & \ + B_nK\begin{bmatrix}x_1+e_{1,0}&x_2+e_{2,0}&\dots&x_n+e_{n,0}\end{bmatrix}^T\\
                &= A_nx + B_nKx+B_nK\begin{bmatrix}e_{1,0}&e_{2,0}&\dots&e_{n,0}\end{bmatrix}^T\\
                &=K^*x+B_nKe_0, 
            \end{aligned}\notag
        \end{equation}
        where we define $e_0 = \begin{bmatrix}e_{1,0}&e_{2,0}&\dots&e_{n,0}\end{bmatrix}^T,$ and $K^* = A_n+B_nK.$
    \end{enumerate}
    As a result of these properties, we compute and organize the error derivatives in a new matrix, as shown in \eqref{eq:err_dots} at the top of the next page.
\begin{figure*}[!t]
\normalsize
\setcounter{MYtempeqncnt}{\value{equation}}
    \begin{equation}\label{eq:err_dots}
        \dot E = \begin{bmatrix}
            -\sigma e_{1,0}-x_2 & -\sigma e_{1,1}-x_3 & \dots & -\sigma e_{1,n-1}-K^*x-B_nKe_0\\
        -\sigma e_{2,0} + 
        \sigma e_{1,1}-x_3 & -\sigma e_{2,1}+\sigma e_{1,2}-x_4 & \dots & 0\\
        \vdots & \vdots & \iddots & \vdots\\
        -\sigma e_{n,0}+\sigma e_{n-1,1}-K^*x-B_nKe_0 & 0 & \dots & 0
        \end{bmatrix}
    \end{equation}
\hrulefill
\vspace*{4pt}
\end{figure*}
Moreover, using the error variables the time derivative of $x$ can be expressed as:
    \begin{equation}
        \begin{aligned}
            \dot x &= Ax+BK\begin{bmatrix}\widehat x_1&\widehat x_2&\dots&\widehat x_n\end{bmatrix}^T\\
            &= Ax+BK\begin{bmatrix}x_1+e_{1,0}&x_2+e_{2,0}&\dots&x_n+e_{n,0}\end{bmatrix}^T\\
            &= (A+BK)x+BK\begin{bmatrix}e_{1,0}&e_{2,0}&\dots&e_{n,0}\end{bmatrix}^T\\
            &= (A+BK)x+BKe_0.
        \end{aligned}\notag
    \end{equation}
    
    Let $e\in\R^{\frac{n(n+1)}{2}}$ denote the vector containing all the error variables $e_{i,j}$, and consider the following Lyapunov function $V:\R^n\times\R^{\frac{n(n+1)}{2}}\to \R_{\geq 0}$ of $x$ and $e$:
    \begin{equation}\label{eq:main_lyapunov}
        V(x, e) = x^TPx + \sum_{i=1}^n\sum_{j=0}^{n-i}e_{i,j}^2,
    \end{equation}
    where $P$ is specified as defined in \eqref{eq:riccati}.
    We proceed to compute the time derivative of~\eqref{eq:main_lyapunov}:
    \begin{equation}\label{eq:lyapunov_dot}
        \begin{aligned}
            \dot V =& \dot x^TP x + x^TP\dot x + 2\sum_{i=1}^n\sum_{j=0}^{n-i}e_{i,j}\dot e_{i,j}\\
            =& \left((A+BK)x+BKe_0\right)^TPx \\
            &+ x^TP\left((A+BK)x+BKe_0\right)+2\sum_{i=1}^n\sum_{j=0}^{n-i}e_{i,j}\dot e_{i,j}\\
            =& -x^TQx+2x^TPBKe_0-2K^*x\sum_{i=1}^{n}e_{i,n-i}\\
            &-2\sigma\left[\sum_{i=1}^n\sum_{j=0}^{n-i}e_{i,j}^2 - \sum_{i=2}^{n}\sum_{j=0}^{n-i}e_{i,j}e_{i-1,j+1}\right]\\
            &-2B_nKe_0\sum_{i=1}^{n}e_{i,n-i}-2\sum_{i=1}^{n-1}\sum_{j=0}^{n-i-1}e_{i,j}x_{i+j+1}.
        \end{aligned}
    \end{equation}
    The first term in the last equality was obtained through \eqref{eq:riccati}, and the last four terms were obtained by expanding $\sum_{i=1}^n\sum_{j=0}^{n-i}e_{i,j}\dot e_{i,j}$ according to~\eqref{eq:errs} and~\eqref{eq:err_dots}.
    With the preceding at hand we can prove that~\eqref{eq:main_lyapunov} is strictly negative definite for sufficiently large values of $\sigma$. In the following we make repeated use of Lemma~\ref{eq:complete_square}\footnote{For space reasons we don't report the full proof, but it can be easily obtained by noting that $(\sqrt{\varepsilon}a\mp \sqrt{\varepsilon}^{-1}b)^T(\sqrt{\varepsilon}a\mp \sqrt{\varepsilon}^{-1}b) = \varepsilon a^Ta+\varepsilon^{-1}b^Tb\mp 2a^Tb \geq 0$.}.
    \begin{lemma}\label{eq:complete_square}
     All triples ${a,b,\varepsilon}$ with $a,b\in\R^n$ and $\varepsilon\in\R_{>0}$ satisfy the inequality:
    \begin{equation}
        \pm 2a^Tb \leq \varepsilon \norm{a}^2 + \frac{1}{\varepsilon}\norm{b}^2.\notag
    \end{equation}
    \end{lemma}
    We proceed by deriving upper bounds for all of the terms in~\eqref{eq:lyapunov_dot}.
    \begin{enumerate}
        \item Since $Q$ is strictly positive definite, we know that:
        \begin{equation}
            -x^TQx \leq -\lambda\norm{x}^2,\notag
        \end{equation}
        where $\lambda$ is the smallest eigenvalue of $Q$.
        
        \item By using Lemma~\eqref{eq:complete_square} with $a=x$ and $b=PBKe_0$, we can write:
            \begin{equation}
    \begin{aligned}
            2x^TPBKe_0&\leq \varepsilon\norm{x}^2+\frac{\norm{PBK}^2}{\varepsilon}\norm{e_0}^2\\
            &\leq \varepsilon\norm{x}^2+\frac{\norm{PBK}^2}{\varepsilon}\sum_{i=1}^n\sum_{j=0}^{n-i}e_{i,j}^2,
    \end{aligned}\notag
    \end{equation}
        for any $\varepsilon>0$.
        
        \item By using Lemma~\eqref{eq:complete_square} with $b=\sum_{i=1}^{n}e_{i,n-i}$ and  $a=K^*x$, we can write:
        \begin{equation}
        \begin{aligned}
            -2K^*x\sum_{i=1}^{n}e_{i,n-i}&\leq \varepsilon\norm{K^*}^2\norm{x}^2 + \frac{1}{\varepsilon}\left(\sum_{i=1}^{n}e_{i,n-i}\right)^2.
        \end{aligned}\notag
        \end{equation}
        Note that $\left(\sum_{i=1}^{n}e_{i,n-i}\right)^2$ is a strictly positive definite quadratic form, therefore there exists a constant $c_1\in\R_{>0}$ such that:
        \begin{equation}
            \left(\sum_{i=1}^{n}e_{i,n-i}\right)^2\leq c_1\sum_{i=1}^{n}e_{i,n-i}^2,\notag
        \end{equation}
        allowing us to write:
        \begin{equation}
        \begin{aligned}
            -2K^*x\sum_{i=1}^{n}e_{i,n-i} &\leq \varepsilon\norm{K^*}^2\norm{x}^2 + \frac{c_1}{\varepsilon}\sum_{i=1}^{n}e_{i,n-i}^2\\
            &\leq \varepsilon\norm{K^*}^2\norm{x}^2 + \frac{c_1}{\varepsilon}\sum_{i=1}^n\sum_{j=0}^{n-i}e_{i,j}^2.
        \end{aligned}\notag
        \end{equation}
        
        \item Analogously to point 3, for some $c_2\in\R_{>0}$ and $\varepsilon=1$ we can write the bound:
        \begin{equation}
        \begin{aligned}
            -2B_nKe_0 & \sum_{i=1}^{n}e_{i,n-i}\\
            &\leq \norm{B_nK}^2\norm{e_0}^2+ \left(\sum_{i=1}^{n}e_{i,n-i}\right)^2\\
            &\leq \norm{B_nK}^2\sum_{i=1}^n\sum_{j=0}^{n-i}e_{i,j}^2+ c_2\sum_{i=1}^n\sum_{j=0}^{n-i}e_{i,j}^2\\
            &= \left(\norm{B_nK}^2+c_2\right)\sum_{i=1}^n\sum_{j=0}^{n-i}e_{i,j}^2.
        \end{aligned}\notag
        \end{equation}
        
        \item By repeatedly applying Lemma~\eqref{eq:complete_square} with $a=x_{i+j+1}$ and $b=e_{i,j}$, we can write:
        \begin{equation}
        \begin{aligned}
            -2\sum_{i=1}^{n-1} & \sum_{j=0}^{n-i-1}e_{i,j}x_{i+j+1}\\
            &\leq \sum_{i=1}^{n-1}\sum_{j=0}^{n-i-1}(\varepsilon x_{i+j+1}^2 + \frac{1}{\varepsilon}e_{i,j}^2)\\
            &= \varepsilon\sum_{i=1}^{n-1}\sum_{j=0}^{n-i-1}x_{i+j+1}^2+ \frac{1}{\varepsilon}\sum_{i=1}^{n-1}\sum_{j=0}^{n-i-1}e_{i,j}^2,
        \end{aligned}\notag
        \end{equation}
        and again there exist constants $c_3,c_4\in\R_{>0}$ allowing us to bound the above positive definite quadratic forms by:
        \begin{equation}
            -2\sum_{i=1}^{n-1}\sum_{j=0}^{n-i-1}e_{i,j}x_{i+j+1} \leq \varepsilon c_3\norm{x}^2 + \frac{c_4}{\varepsilon}\sum_{i=1}^n\sum_{j=0}^{n-i}e_{i,j}^2.\notag
        \end{equation}
        \item The remaining term can be written as:
        \begin{equation}
        \begin{aligned}
            &-2\sigma \left[\sum_{i=1}^n\sum_{j=0}^{n-i}e_{i,j}^2 - \sum_{i=2}^{n}\sum_{j=0}^{n-i}e_{i,j}e_{i-1,j+1}\right]\\
            & \ = 2\sigma\sum_{k=1}^n\left(-\sum_{\ell=0}^{k-1}e_{k-\ell,\ell}^2+\sum_{\ell=0}^{k-2}e_{k-\ell,\ell}e_{k-\ell-1,\ell+1}\right),
        \end{aligned}\notag
        \end{equation}
        where, through a slight abuse of notation, we define the sum $\sum\limits^j_i 1$ to be zero whenever $j<i.$ Intuitively, the inner sums for specific values of $k$ correspond to summing the elements of $E\odot \dot E$ along its anti-diagonals including only the terms that are multiplied by $\sigma$, e.g., $k=1$ corresponds to the first anti-diagonal (the top left element), $k=2$ to the second anti-diagonal, etc.
        
        We claim that for any $k$ there exists a constant $d_k\in\R_{>0}$ such that:
        \begin{equation}
        \label{ineq:negativedef}
            -\sum_{\ell=0}^{k-1}e_{k-\ell,\ell}^2+\sum_{\ell=0}^{k-2}e_{k-\ell,\ell}e_{k-\ell-1,\ell+1} \leq -d_k\sum_{\ell=0}^{k-1}e_{k-\ell,\ell}^2.
        \end{equation}
        
        This can be shown by first defining $z_{\ell} = e_{k-\ell,\ell}$ and applying the following lemma\footnote{For space reasons we won't include the proof, but it can be easily obtained by rewriting the quadratic form as $-z^TMz$, with $M\in\R^{k\times k}$ a symmetric matrix and noting that $M$ can be reduced to upper triangular form with strictly negative elements on the diagonal by using row operations. This implies that $\det (M) < 0$ and since $M$ is negative semidefinite by Gershgorin's circle theorem, it's enough to prove strict negative definiteness of $M$}.
        \begin{lemma}\label{lem:posdef}
    For any $k\in\N_{>0}$, the quadratic form:
    \begin{equation}
        -\sum_{i=1}^{k}z_i^2+\sum_{i=1}^{k-1}z_iz_{i+1}\notag
    \end{equation}
    is negative definite.
    \end{lemma}
    Finally, inequality \eqref{ineq:negativedef} is reached by noting that the negative definiteness of the left terms implies the existence of $d_k\in\R_{>0}$ such that the inequality holds. Inequality \eqref{ineq:negativedef} allows us to write:
    \begin{equation}
        \begin{aligned}
            -2\sigma &\left[\sum_{i=1}^n\sum_{j=0}^{n-i}e_{i,j}^2 - \sum_{i=2}^{n}\sum_{j=0}^{n-i}e_{i,j}e_{i-1,j+1}\right]\\
            &= 2\sigma\sum_{k=1}^n\left(-\sum_{\ell=0}^{k-1}e_{k-\ell,\ell}^2+\sum_{\ell=0}^{k-2}e_{k-\ell,\ell}e_{k-\ell-1,\ell+1}\right)\\
            &\leq -2\sigma\sum_{k=1}^n d_k\sum_{\ell=0}^{k-1}e_{k-\ell,\ell}^2 \leq -2\sigma d\sum_{k=1}^n\sum_{\ell=0}^{k-1}e_{k-\ell,\ell}^2\\
            &= -2\sigma d\sum_{i=1}^n\sum_{j=0}^{n-i}e_{i,j}^2,
        \end{aligned}\notag
    \end{equation}
    where $\sum_{i=1}^n\sum_{j=0}^{n-i}e_{i,j}^2 = \sum_{k=1}^n\sum_{\ell=0}^{k-1}e_{k-\ell,\ell}^2,$ and  $d=\min\{d_1,d_2,\dots,d_n\}.$
    \end{enumerate}
    We can now provide the full upper bound for the time derivative of $V$:
    \begin{equation}
        \begin{aligned}
            \dot V &\leq -x^TQx+2x^TPBKe_0-2K^*x\sum_{i=1}^{n}e_{i,n-i}\\
            &\quad-2\sigma\left[\sum_{i=1}^n\sum_{j=0}^{n-i}e_{i,j}^2 - \sum_{i=2}^{n}\sum_{j=0}^{n-i}e_{i,j}e_{i-1,j+1}\right]\\
            &\quad-2B_nKe_0\sum_{i=1}^{n}e_{i,n-i}-2\sum_{i=1}^{n-1}\sum_{j=0}^{n-i-1}e_{i,j}x_{i+j+1}\\
            & \leq -\lambda\norm{x}^2 + \varepsilon\norm{x}^2+\frac{\norm{PBK}^2}{\varepsilon}\sum_{i=1}^n\sum_{j=0}^{n-1}e_{i,j}^2\\
            &\quad +\varepsilon\norm{K^*}^2\norm{x}^2 + \frac{c_1}{\varepsilon}\sum_{i=1}^n\sum_{j=0}^{n-1}e_{i,j}^2\\
            &\quad +\left(\norm{B_nK}^2+c_2\right)\sum_{i=1}^n\sum_{j=0}^{n-1}e_{i,j}^2+\varepsilon c_3\norm{x}^2 \\
            &\quad + \frac{c_4}{\varepsilon}\sum_{i=1}^n\sum_{j=0}^{n-1}e_{i,j}^2 -2\sigma d\sum_{i=1}^n\sum_{j=0}^{n-1}e_{i,j}^2\\
            &\leq \left(-\lambda + \varepsilon\left(1+c_3+\norm{K^*}^2\right)\right)\norm{x}^2 \\
            &\quad + \frac{1}{\varepsilon}\left(\norm{PBK}^2+c_1+c_4\right)\sum_{i=1}^n\sum_{j=0}^{n-1}e_{i,j}^2\\
            &\quad + \left(\norm{B_nK}^2+c_2-2\sigma d\right)\sum_{i=1}^n\sum_{j=0}^{n-1}e_{i,j}^2.
        \end{aligned}\notag
    \end{equation}
    The above expression is negative definite as long as:
    \begin{equation}
    \begin{aligned}
             \varepsilon &< \frac{\lambda}{1+c_3+\norm{K^*}^2}\\
             \sigma &> \underline\sigma = \frac{\varepsilon^{-1}\left(\norm{PBK}^2+c_1+c_4\right)+\norm{B_nK}^2+c_2}{2d}.
    \end{aligned}\notag
    \end{equation}
    
    We note that negative definiteness of $\dot V$ (in the state $x$ and the error variables $e$) implies $\lim_{t\to\infty} x = 0$ and $\lim_{t\to\infty} e = 0$. However, we know from~\eqref{ErrorDef} that the estimates $\widehat x_1, \widehat x_2, \dots, \widehat x_n$ can be represented as linear combinations of the state $x$ and the errors $e$, therefore we conclude that:
    \begin{equation}
        \forall i\in\{1,2,\dots,n\}:\quad\lim_{t\to\infty}\widehat x_i = 0,\notag
    \end{equation}
    proving asymptotic stability for the original system~\eqref{MainSystem}.
\end{proof}

\subsection{Proof of Theorem \eqref{theo:main_ext}}
\begin{theorem}
    There always exists $\underline\gamma\in\R_{>0}$, such that for any $\gamma\in(\underline\gamma, \infty)$ there exists $\underline\sigma\in\R_{>0}$ such that for any $\sigma\in(\underline\sigma, \infty)$ the LTI dynamical system~\eqref{MainSystem2} is asymptotically stable. 
\end{theorem}
\begin{proof}
    The proof proceeds largely analogously to the proof of Theorem~\ref{theo:main}. As before, we introduce the $\frac{n(n+1)}{2}$ error variables as defined in \eqref{ErrorDef}, that we can organize into matrix \eqref{eq:errs}.

    We define $e_0 = \begin{bmatrix}e_{1,0}&e_{2,0}&\dots&e_{n-1,0}\end{bmatrix}^T$, and introduce a new error variable $e_u$, defined as:
    \begin{equation}\label{eq:err_u}
        e_u = x_n - Kx_{1:n-1},
    \end{equation}
    that allows us to rewrite $\dot x_n$ as:
    \begin{equation}
    \begin{aligned}
        \dot x_n =& A_n \dot x_{1:n-1}-\beta\gamma\left(\widehat x_n - K\begin{bmatrix}\widehat x_1&\widehat x_2&\dots&\widehat x_{n-1}\end{bmatrix}^T\right)\\
        =&A_n \dot x_{1:n-1} -\beta\gamma (e_u+e_n-Ke_0),
    \end{aligned}\notag
    \end{equation}
    and the derivatives of $x_1,x_2,\dots,x_{n-1}$ as:
    \begin{equation}
        \begin{aligned}
            \dot x_{1:n-1} = Ax_{1:n-1}+Bx_n
            = (A+BK)x_{1:n-1}+Be_u,
        \end{aligned}\notag
    \end{equation}
    where we used~\eqref{eq:err_u} to obtain the second equality.
    
    As in~\eqref{eq:err_dot}, the time derivatives of the error variables are:
    \begin{equation}
        \begin{aligned}
            \dot e_{i,j} &= -\sigma(\widehat x_i^{(j)}-\widehat x_{i-1}^{(j+1)})-\dot x_{i+j},
        \end{aligned}\notag
    \end{equation}
    with the same properties as in Theorem~\ref{theo:main}, except for point 4 which now becomes:
    \begin{enumerate}
        \setcounter{enumi}{3}
        \item If $i+j=n$:
        \begin{equation}
            \dot x_{i+j} = \dot x_n =A_n \dot x_{1:n-1} -\beta\gamma e_u-\beta\gamma(e_n-Ke_0),\notag
        \end{equation}
    \end{enumerate}
    As a result, we again compute and organize the error derivatives in matrix form for clarity, as shown in \eqref{eq:err_dots_ext} at the top of the next page.
\begin{figure*}[!t]
\footnotesize
\setcounter{MYtempeqncnt}{\value{equation}}
    \begin{equation}\label{eq:err_dots_ext}
        \dot E = \begin{bmatrix}
            -\sigma e_{1,0}-x_2 & -\sigma e_{1,1}-x_3 & \dots & -\sigma e_{1,n-1}+A_n \dot x_{1:n-1}-\beta\gamma (e_ue_n-Ke_0)\\
        -\sigma e_{2,0} + 
        \sigma e_{1,1}-x_3 & -\sigma e_{2,1}+\sigma e_{1,2}-x_4 & \dots & 0\\
        \vdots & \vdots & \iddots & \vdots\\
        -\sigma e_{n,0}+\sigma e_{n-1,1}+A_n \dot x_{1:n-1}-\beta\gamma (e_u+e_n-Ke_0) & 0 & \dots & 0
        \end{bmatrix}
    \end{equation}
\hrulefill
\vspace*{4pt}
\end{figure*}
\normalsize
Finally, we compute the derivative of $e_u$:
    \begin{equation}
        \begin{aligned}
            \dot e_u =& \dot x_n - K\dot x_{1:n-1}\\
           =& -\beta\gamma \left(e_u+e_n-Ke_0\right)\\
           &-(K-A_n)\left((A+BK)x_{1:n-1}+Be_u\right)\\
             =& -\beta\gamma \left(e_u+e_n-Ke_0\right)\\
             &-\overline{K}\left((A+BK)x_{1:n-1}+Be_u\right),
        \end{aligned} \notag
    \end{equation}
    where we define $\overline{K} = K-A_n.$ We now consider the following Lyapunov function $V:\R^{n-1}\times\R\times\R^{\frac{n(n+1)}{2}}\to \R_{\geq 0}$ of $x$ and all the error variables, denoting the vector containing all the errors except $e_u$ by $e\in\R^{\frac{n(n+1)}{2}}$:
    \begin{equation}\label{eq:main_lyapunov_ext}
        V(x, e_u, e) = x_{1:n-1}^TPx_{1:n-1} + e_u^2 + \sum_{i=1}^n\sum_{j=0}^{n-i}e_{i,j}^2.
    \end{equation}
    Computing the time derivative of~\eqref{eq:main_lyapunov_ext} results in:
    \begin{equation}\label{eq:lyapunov_dot_ext}
        \begin{aligned}
            \dot V &= 2x_{1:n-1}^TP\dot x_{1:n-1} + 2e_u\dot e_u + 2\sum_{i=1}^n\sum_{j=0}^{n-i}e_{i,j}\dot e_{i,j}\\
            &=-x_{1:n-1}^TQx_{1:n-1}+2x_{1:n-1}^TPBe_u+2(-\beta\gamma-\overline{K}B)e_u^2\\
            &-2e_u\beta\gamma(e_n-Ke_0)-2e_u\overline{K}(A+BK)x_{1:n-1}\\
            &-2\sigma\left[\sum_{i=1}^n\sum_{j=0}^{n-i}e_{i,j}^2 - \sum_{i=2}^{n}\sum_{j=0}^{n-i}e_{i,j}e_{i-1,j+1}\right]\\
            &-2\beta\gamma e_u\sum_{i=1}^{n}e_{i,n-i}-2\beta\gamma(e_n-Ke_0)\sum_{i=1}^{n}e_{i,n-i}\\
            &-2A_n(A+BK)x_{1:n-1}\sum_{i=1}^{n}e_{i,n-i}-2A_nBe_u\sum_{i=1}^{n}e_{i,n-i}\\
            &-2\sum_{i=1}^{n-1}\sum_{j=0}^{n-i-1}e_{i,j}x_{i+j+1}.
        \end{aligned}
    \end{equation}
    
    Just like in Theorem~\ref{theo:main}, we proceed by deriving upper bounds for all of the terms in~\eqref{eq:lyapunov_dot_ext}. We make use of Lemma~\eqref{eq:complete_square} every time we encounter a mixed product of variables.
    \begin{enumerate}
        \item Denoting the smallest eigenvalue of $Q$ by $\lambda$:
        \begin{equation}
            -x_{1:n-1}^TQx_{1:n-1} \leq -\lambda\norm{x_{1:n-1}}^2.\notag
        \end{equation}
        
        \item For any $\varepsilon\in\R_{>0}$:
        \begin{equation}
            2x_{1:n-1}^TPBe_u\leq \varepsilon\norm{x_{1:n-1}}^2+\frac{\norm{PB}^2}{\varepsilon}e_u^2.\notag
        \end{equation}
        
        \item For some constant $c_1\in\R_{>0}$:
        \begin{equation}
        \begin{aligned}
            -2e_u\beta\gamma(e_n-Ke_0) &\leq e_u^2+\beta^2\gamma^2(e_n-Ke_0)^2\\
            &\leq e_u^2+ \beta^2\gamma^2c_1\sum_{i=1}^n\sum_{j=0}^{n-1}e_{i,j}^2.
        \end{aligned}\notag
            \end{equation}
        
        \item For any $\varepsilon\in\R_{>0}$:
        \begin{equation}
        \begin{aligned}
        -2e_u\overline{K}(A+BK)&x_{1:n-1}\leq\\
        &\varepsilon\norm{x_{1:n-1}}^2+\frac{\norm{\overline{K}(A+BK)}^2}{\varepsilon}e_u^2.
        \end{aligned}\notag
        \end{equation}
        
        \item For some constant $c_2\in\R_{>0}$:
        \begin{equation}
        \begin{aligned}
                      -2\beta\gamma e_u\sum_{i=1}^{n}e_{i,n-i} &\leq e_u^2 +\beta^2\gamma^2\left(\sum_{i=1}^{n}e_{i,n-i}\right)^2\\
                      &\leq e_u^2 +\beta^2\gamma^2c_2\sum_{i=1}^n\sum_{j=0}^{n-1}e_{i,j}^2. 
        \end{aligned}\notag
        \end{equation}
        
        \item For some constant $c_3\in\R_{>0}$:
        \begin{equation}
            -2\beta\gamma(e_n-Ke_0)\sum_{i=1}^{n}e_{i,n-i} \leq \beta\gamma c_3\sum_{i=1}^n\sum_{j=0}^{n-1}e_{i,j}^2.\notag
        \end{equation}
        
        \item For some $c_4\in\R_{>0}$ and any $\varepsilon\in\R_{>0}$:
        \begin{equation}
        \begin{aligned}
            -2&A_n(A+BK)x_{1:n-1}\sum_{i=1}^{n}e_{i,n-i} \\
            &\leq 2\norm{A_n(A+BK)}\norm{x_{1:n-1}}\norm{\sum_{i=1}^{n}e_{i,n-i}}\\
            &\leq \varepsilon\norm{x_{1:n-1}}^2 + \frac{\norm{A_n(A+BK)}^2}{\varepsilon}\left(\sum_{i=1}^{n}e_{i,n-i}\right)^2\\
            &\leq \varepsilon\norm{x_{1:n-1}}^2 + \frac{c_4}{\varepsilon}\sum_{i=1}^n\sum_{j=0}^{n-1}e_{i,j}^2
        \end{aligned}\notag
        \end{equation}
        
        \item For some constant $c_5\in\R_{>0}$:
        \begin{equation}
        \begin{aligned}
                      -2A_nBe_u\sum_{i=1}^{n}e_{i,n-i} &\leq e_u^2 +(A_nB)^2\left(\sum_{i=1}^{n}e_{i,n-i}\right)^2\\
                      &\leq e_u^2 +(A_nB)^2c_5\sum_{i=1}^n\sum_{j=0}^{n-1}e_{i,j}^2. 
        \end{aligned}\notag
        \end{equation}
        
        \item For some $c_6,c_7\in\R_{>0}$ and any $\varepsilon\in\R_{>0}$:
        \begin{equation}
        \begin{aligned}
            -2&\sum_{i=1}^{n-1}\sum_{j=0}^{n-i-1}e_{i,j}x_{i+j+1}\\
            =&-2\sum_{i=1}^{n-1}\sum_{j=0}^{n-i-2}e_{i,j}x_{i+j+1} -2x_{n}\sum_{i=1}^{n-1}e_{i,n-i-1}\\
            =& -2\sum_{i=1}^{n-1}\sum_{j=0}^{n-i-2}e_{i,j}x_{i+j+1}\\
            &-2(e_n+Kx_{1:n-1})\sum_{i=1}^{n-1}e_{i,n-i-1}\\
            \leq& \varepsilon c_6\norm{x_{1:n-1}}^2 + \frac{c_7}{\varepsilon}\sum_{i=1}^n\sum_{j=0}^{n-i}e_{i,j}^2.
        \end{aligned}\notag
        \end{equation}
        Note that all $x$ terms in the right-hand side of the second equality have indices between $1$ and $n-1$.
        
        \item Exactly as in Theorem~\eqref{theo:main}, for some $d\in\R_{>0}$:
        \begin{equation}
        \begin{aligned}
            -2\sigma&\left[\sum_{i=1}^n\sum_{j=0}^{n-i}e_{i,j}^2 - \sum_{i=2}^{n}\sum_{j=0}^{n-i}e_{i,j}e_{i-1,j+1}\right] \\
            &\leq -2\sigma d\sum_{i=1}^n\sum_{j=0}^{n-1}e_{i,j}^2.
        \end{aligned}\notag
        \end{equation}
    \end{enumerate}
    
    Substituting all of these bounds in~\eqref{eq:lyapunov_dot_ext}, we can bound $\dot V$ by:
    \begin{equation}
        \begin{aligned}
            \dot V \leq& 
            \left(-\lambda+\varepsilon(3+c_6)\right)\norm{x_{1:n-1}}^2\\
            &+\left(\frac{\norm{PB}^2\norm{\overline{K}(A+BK)}^2}{\varepsilon}-2\beta\gamma-2\overline{K}B+3\right)e_u^2\\
            &+\left(\beta^2\gamma^2(c_1+c_2)+\beta\gamma c_3+\frac{c_4+c_7}{\varepsilon}\right)\sum_{i=1}^n\sum_{j=0}^{n-1}e_{i,j}^2\\
            &+\left((A_nB)^2c_5-2\sigma d\right)\sum_{i=1}^n\sum_{j=0}^{n-1}e_{i,j}^2.
        \end{aligned}\notag
    \end{equation}
    The previous expression is negative definite as long as:
    \begin{equation}
        \begin{aligned}
            \varepsilon &< \frac{\lambda}{3+c_6},\\
            \gamma &> \underline\gamma = \frac{1}{2\beta}\left(\frac{\norm{PB}^2\norm{\overline{K}(A+BK)}^2}{\varepsilon}-2\overline{K}B+3\right),\\
            \sigma &> \underline\sigma \\
            &= \frac{1}{2d}\left(\beta^2\gamma^2(c_1+c_2)+\beta\gamma c_3+\frac{c_4+c_7}{\varepsilon}+(A_nB)^2c_5\right).
        \end{aligned}\notag
    \end{equation}
    
    To finish the proof, we note that negative definiteness of $\dot V$ (in the state $x$ and the error variables $e$ (except $e_u$)) implies $\lim_{t\to\infty} x_{1:n-1} = 0$, $\lim_{t\to\infty} e_u = 0$ and $\lim_{t\to\infty} e = 0$. However, both $x_n$ and the estimates $\widehat x_1,\widehat x_2,\dots,\widehat x_n$ are linear combinations of the state $x_{1:n-1}$, $e_u$, and the errors $e$, therefore we conclude that:
    \begin{equation}
    \notag
        \forall i\in\{1,2,\dots,n\}:\quad\lim_{t\to\infty}\widehat x_i = 0 \ \text{and} \ \lim_{t\to\infty}x_n = 0,
    \end{equation}
    proving asymptotic stability for the original system~\eqref{MainSystem2}.
\end{proof}

\bibliographystyle{alpha}
\bibliography{Bibliography}
\end{document}